\documentclass[11pt]{article}

\usepackage{amssymb}
\usepackage{amsthm}
\usepackage{amsmath}
\usepackage{amssymb}
\usepackage{amsfonts}
\usepackage{multirow}
\usepackage{amsthm}
\usepackage{authblk}
\usepackage{xcolor}
\usepackage{graphicx}
\usepackage{slashbox}
\usepackage{natbib}
\usepackage[a4paper,textwidth=15cm]{geometry}
\usepackage{tikz}
\usetikzlibrary{plotmarks}

\newtheorem{theorem}{Theorem}
\newtheorem{lemma}{Lemma}
\newtheorem{proposition}{Proposition}
\newtheorem{corollary}{Corollary}

\theoremstyle{definition}

\newtheorem{definition}{Definition}

\newcommand{\R}{\mathbb{R}}

\newcommand{\E}{\mathbb{E}}
\newcommand{\G}{\mathbb{G}}
\newcommand{\F}{\mathbb{F}}

\title{Transform orders and stochastic monotonicity of statistical functionals}

\author[1]{Tommaso Lando\thanks{Tommaso Lando was supported by the Italian funds ex MURST 60\% 2021, by the Czech Science Foundation (GACR) under project 20-16764S and moreover by SP2022/4, an SGS research project of V\v{S}B-TU Ostrava. The support is greatly acknowledged.\\
\hspace*{.6cm}E-mail: tommaso.lando@unibg.it; Corresponding author.}}
\affil[1]{Department of Economics, University of Bergamo, Italy; Department of Finance, V\v{S}B-TU Ostrava, Czech Republic}

\author[2]{Idir Arab\thanks{Idir Arab and Paulo Eduardo Oliveira were partially supported by the Centre for Mathematics of the University of Coimbra - UIDB/00324/2020, funded by the Portuguese Government through FCT/MCTES.\\
\hspace*{.6cm}E-mail: idir.bhh@gmail.com}}
\affil[2]{CMUC, Department of Mathematics, University of Coimbra, Portugal}

\author[2]{Paulo Eduardo Oliveira\thanks{E-mail: paulo@mat.uc.pt}}

\date{}

\begin{document}

%

\maketitle

\begin{abstract}
In some inferential statistical methods, such as tests and confidence intervals, it is important to describe the stochastic behavior of statistical functionals, aside from their large sample properties. We study such a behavior in terms of the usual stochastic order. For this purpose, we introduce a generalized family of stochastic orders, which is referred to as transform orders, showing that it provides a flexible framework for deriving stochastic monotonicity results. Given that our general definition makes it possible to obtain some well known ordering relations as particular cases, we can easily apply our method to different families of functionals. These include some prominent inequality measures, such as the generalized entropy, the Gini index, and its generalizations. We also illustrate the applicability of our approach by determining the least favorable distribution, and the behavior of some bootstrap statistics, in some goodness-of-fit testing procedures.
\end{abstract}

%
%
%

\textbf{Keywords}: Gini index, hazard function, inequality, nonparametric test, stochastic dominance.

\section{Introduction}

In statistics, one is often interested in estimating an unknown characteristic of a given distribution $F$, rather than the distribution itself. In many such cases, these characteristics may be represented by some \textit{probability functional} $T(F):\mathcal{F}\rightarrow \R$, where $\mathcal{F}$ is the space of cumulative distribution functions (CDFs). The most intuitive way of estimating $T(F)$ is the \textit{plug-in} method, which simply consists in replacing the unknown $F$ with its natural estimator, namely, the empirical CDF $\F_n$. Correspondingly, we shall refer to $T(\F_n)$ as a \textit{statistical functional}. However, since the empirical distribution is a random process, that is, a function of the random sample, then statistical functionals are random variables, and therefore it is important to study their behavior from a stochastic point of view. While most of the work in the literature focuses on large sample properties of statistical functionals, because of their applications to estimation theory, we are concerned with a different problem, namely, we are interested in establishing relationships between their distribution functions that hold for any given sample, regardless of its size.

Let $T$ be some functional, and let $\F_n$ and $\G_n$ be the empirical CDFs that correspond to random samples from $F$ and $G$, respectively.
We are interested in determining the conditions on $F$, $G$ and $T$, such that $T(\F_n)$ \textit{stochastically dominates} $T(\G_n)$ with respect to the \textit{usual stochastic order} (recalled below in Definition~\ref{orders}). This is a relevant information as it may be used to help characterizing finite sample properties of testing procedures, as we will be describing later. Results of this kind may be rather complicated to derive, depending on the mathematical form of $T$, and on the initial conditions on the baseline distributions $F$ and $G$. In this regard, \cite{greenwood} recently studied the stochastic dominance behavior of the Greenwood statistic \citep{greenwood1946} when the baseline distributions are comparable in terms of the \textit{star order}. We will show that this result is a particular case of a much more general one, making it possible to establish dominance relations within a wider family of statistical functionals, specifically, functionals that are compatible with a general class of \textit{transform orders}, generated by some family of reference functions (see Definition~\ref{transf order} below).

Our main result, Theorem~\ref{teo}, presented in Section~\ref{sec:2}, shows that, if the baseline distributions are ordered with respect to some transform order, the statistical functionals are stochastically ordered, provided that the corresponding probability functional is order-preserving with respect to the same transform order. This result has a wide range of applicability, as our definition of transform order includes the usual stochastic order, the convex transform order, the star order, the superadditive and the dispersive orders \citep{shaked2007}. Therefore, we are able to derive stochastic properties of several important families of functionals.

As an immediate application of Theorem~\ref{teo}, in Section~\ref{sec:3} we obtain the stochastic monotonicity of many important measures of inequality, including the generalized entropy class \citep{shorrocks1980,shorrocks}, the well known Gini index \citep{gini} and some of its generalizations \citep{mehran,donaldson}.
Differently, in Section~\ref{sec:tests}, we use Theorem~\ref{teo} to determine, in a simple way, the least favorable distribution, as well as the behavior of some bootstrap statistics, for some goodness-of-fit testing problems, related to the convex order or to the star order, which are relevant in areas such as survival analysis, reliability, and shape-constrained inference \citep{barlow1971}.

\section{Main results}
\label{sec:2}

We begin by establishing some notations. The random variables $X$ and $Y$ have CDFs $F$ and $G$ and supports $S_F$ and $S_G$, respectively. Let us denote with $F^{-1}$ the left continuous generalized inverse, namely, the quantile function, of $F$. We recall that a stochastic order is a binary relation $\succ$ over $\mathcal{F}$ that is reflexive and transitive. In particular, observe that $\succ$ does not generally satisfy the antisymmetry property, that is, $F\succ G$ and $G\succ F$ do not necessarily imply $X=_dY$, and it is generally not total. A functional $T$ is said to be \textit{isotonic}, or order-preserving, with respect to $\succ$, whenever, for every pair $F,G\in \mathcal{F}$ such that $F\succ G$, it holds that $T(F)\geq T(G)$. Let us denote by $\mathcal{T}$ the class of functionals defined on $\mathcal{F}$. Then, we may represent by $\mathcal{I}(\succ)$ the class of functionals that are isotonic with respect to $\succ$, that is,
\begin{equation*}
\mathcal{I}(\succ)=\{T\in\mathcal{T}:F\succ G\implies T(F)\geq T(G),\;\forall F,G\in\mathcal{F} \}.
\end{equation*}
Likewise, the class of functionals that are \textit{antitonic}, or order-reversing, with respect to $\succ$, is denoted as $\mathcal{A}(\succ)=\{T\in\mathcal{T}:F\succ G\implies T(F)\leq T(G),\;\forall F,G\in\mathcal{F} \}$. It is worth noticing that, given a pair of stochastic orders, say $\succ_0$ and $\succ_1$, if $F\succ_0{G}\implies{F}\succ_1{G}$, then $\mathcal{I}(\succ_1)\subset \mathcal{I}(\succ_0)$.

Given that stochastic orders and probability functionals depend only on the distribution functions of the random variables, we may write $F\succ G$ or $X\succ Y$, and $T(F)$ or $T(X)$, interchangeably.

\subsection{A general family of orders}We now introduce the following definition of a transform order, generated by a family of functions $\mathcal{C}$, whose choice enables the generalization of some well known stochastic orders.
As usual, $f|_E$ denotes the restriction of a function $f$ to the set $E$.

\begin{definition}\label{transf order}
Let $\mathcal{C}$ be some family of non-decreasing functions. We say that $X$ dominates $Y$ with respect to the $\mathcal{C}$-transform order, denoted by $X\geq_\mathcal{C} Y$, if $F^{-1}\circ G|_{S_G}\in \mathcal{C}$.
\end{definition}

Note that $X=_d F^{-1}\circ G (Y)$, so, basically, a transform order holds when the function that transforms (distributionally) one random variable into the other satisfies some properties of interest. Such properties fully characterize the dominance relations.
It is easy to see that $\geq_\mathcal{C}$ fulfills the basic properties of stochastic orders, and, moreover, $\mathcal{C}_0\subset\mathcal{C}_1$ means that $F\geq_{\mathcal{C}_0}G\implies F\geq_{\mathcal{C}_1}G$. As limiting cases, if $\mathcal{C}$ is the class of non-decreasing functions, then $\geq_\mathcal{C}$ is always verified, whereas if $\mathcal{C}$ contains just the identity function then $\geq_\mathcal{C}$ coincides with equality in distribution, $=_d$. The usual stochastic order, generally defined by the relation $F(x)\leq G(x)$, for every $x\in \R$, \citep{shaked2007}, is a transform order, obtained by choosing $\mathcal{C}$ as the class of functions $\psi$ such that $\psi(x)\geq x$. Other well known transform orders my be obtained by appropriately choosing
$\mathcal{C}$: the classes of convex, star-shaped, or superadditive functions, yield the \textit{convex transform order}, \textit{the star order}, or the \textit{superadditive order}, respectively. Also the \textit{dispersive order} can be obtained by choosing $\mathcal{C}$ as the class of functions $\psi$ such that $\psi(x)-x$ is nondecreasing \citep{shaked2007}. Since these orders are particularly important, we recall their definitions below (remember that $F^{-1}\circ G$ is always non-decreasing by construction). However, Definition~\ref{transf order} above is quite general and enables the possibility of defining new stochastic orders. 
\begin{definition}
\label{orders}
We say that $X$ dominates $Y$ with respect to
\begin{enumerate}
\item the usual stochastic order, denoted by $X\geq_{st}Y$, if $F^{-1}\circ G|_{S_G}(x)\geq x$, or equivalently, $F(x)\leq G(x)$, for every $x\in\R$;

\item
the convex transform order, denoted by $X\geq_cY$, if {$F^{-1}\circ G|_{S_G}$ is convex};

\item
the star order, denoted by $X\geq_*Y$, if {$F^{-1}\circ G|_{S_G}$ is star-shaped};

\item
the superadditive order, denoted by $X\geq_{su}Y$, if {$F^{-1}\circ G|_{S_G}$ is superadditive};

\item
the dispersive order, denoted by $X\geq_{disp}Y$, if {$(F^{-1}\circ G(x)-x)|_{S_G}$ is increasing}.

\end{enumerate}
\end{definition}
For non-negative random variables, the following relations hold:
\begin{equation*}
X\geq_c Y\implies X\geq_* Y \iff  \log X\geq_{disp} \log Y \implies X\geq_{su}Y.
\end{equation*}
In the sequel, we will further extend Definition \ref{orders} to enable comparisons of (discrete) empirical distributions, which is necessary to obtain the desired stochastic monotonicity between statistical functionals.

\subsection{Stochastic monotonicity results}
We are concerned with the problem of ranking statistical functionals in terms of the usual stochastic order. Therefore, we are interested in determining the conditions on $F$, $G$ and $T$ such that {$P(T(\F_n)\leq x)\leq P(T(\G_n)\leq x)$} holds {for every $x\in\R$}.
Results of this kind are not predictable in general. In this regard, intuitively, one may wonder whether $T(F)\geq T(G)$ is sufficient for $T(\F_n)\geq_{st} T(\G_n)$, or whether $X\succ Y$ and $T\in\mathcal{I}(\succ)$ imply $T(\F_n)\geq_{st} T(\G_n)$, for a general order $\succ$. It is important to remark that, in general, this is not true (see Subsection~\ref{sec:useful} below).

Our main result, Theorem~\ref{teo}, shows that transform orders provide a general framework for deriving stochastic monotonicity of statistical functionals of the form $T(\F_n)$, conditionally on the isotonicity properties of $T$. However, before proceeding, we need some further preliminary definitions and notations. Indeed, some transform orders may be not defined for realizations, $F_n$ and $G_n$, of the empirical distribution functions $\F_n$ and $\G_n$, as these realizations are discrete. For instance, in the empirical case, the composition $F^{-1}_n\circ G_n$ is a step function, hence neither convex nor star-shaped, so these CDFs are not comparable in terms of $\geq_c$ or $\geq_*$. Therefore, to prove our main results, we must \textit{extend}, when needed, any transform order to the class of observed empirical distributions, according to the following general definition.

\begin{definition}\label{general}
Let $\mathcal{C}$ be some family of non-decreasing functions. We say that $X$ dominates $Y$ with respect to the extended $\mathcal{C}$-transform order, denoted as $X\geq^{e}_{\mathcal{C}} Y$, if there exists a function $\phi\in \mathcal{C}$ such that $\phi|_{S_G}=F^{-1}\circ G|_{S_G}$.
\end{definition}

Note the distinction between Definitions~\ref{transf order} and \ref{general}: the former requires that $F^{-1}\circ G$ satisfies some property, described by $\mathcal{C}$, on the support of $G$, while the latter assumes the existence of some function $\phi\in\mathcal{C}$, which coincides with the composition $F^{-1}\circ G$ on the support of $G$, formally, $\phi|_{S_G}=F^{-1}\circ G|_{S_G}$. This is useful because, for instance, notions such as convexity and star-shapedness are generally not defined for functions whose support is a discrete set of points. In such cases, Definition~\ref{general} makes it possible to consider $\mathcal{C}$ as the class of convex or star-shaped functions, even when $F$ and $G$ are empirical distributions, which is out of the scope of Definition~\ref{transf order}. Indeed, with regard to the pair $F_n$, $G_n$, where $S_{G_n}=\{y_1,\ldots,y_n\}$, the ordering $F_n{\geq^e_{\mathcal{C}}} G_n$ holds if there exists a function $\phi\in \mathcal{C}$ such that $\phi(y_i)=F^{-1}_n\circ G_n(y_i)$, $i=1,\ldots,n$. Moreover, note that $\geq_\mathcal{C}^e$ is a weaker version of  $\geq_\mathcal{C}$, that is, $X\geq_\mathcal{C}Y$ implies $X\geq_\mathcal{C}^e Y$.

We may now establish our main result. Consider a sample $y_1,\ldots,y_n$ from $Y$. As $X=_d F^{-1}\circ G(Y)$, the values ${x_i^\ast=}F^{-1}\circ G(y_i)$, $i=1,\ldots,n$, may be seen as observations from $X$. Now, let us denote with $G_n$ the realization of $\G_n$ corresponding to the sample $y_1,\ldots,y_n$, and with $F^\ast_n$ the particular realization of $\F_n$ corresponding to the sample $x_1^\ast,\ldots,x_n^\ast$. Then, we obtain the following sufficient conditions for $T(\F_n)\geq_{st} T(\G_n)$.

\begin{theorem}\label{teo}\
\begin{enumerate}\item If $X\geq_\mathcal{C}^eY$ and $T\in\mathcal{I}(\geq^e_{\mathcal{C}})$, then $T(\F_n)\geq_{st} T(\G_n)$.
\item If $X\geq_\mathcal{C}^eY$ and $T(F_n^*)\geq T(G_n)$, then $T(\F_n)\geq_{st} T(\G_n)$.
\end{enumerate}
\end{theorem}
\begin{proof}
Let $\tilde{X}_n$ and $\tilde{Y}_n$ be two random variables whose distributions are $F^\ast_n$ and $G_n$, respectively. Obviously, $F^{-1}\circ G(\tilde{Y}_n)=_d\tilde{X}_n$. If $X\geq_\mathcal{C}^eY$, then there exists some $\phi\in\mathcal{C}$ satisfying Definition~\ref{general}. Let $Z$ be a random variable with distribution $H$ and support $S_H$ included in $S_G$. $X\geq_\mathcal{C}^e Y$ implies
\begin{equation*}\phi|_{S_H}=(\phi|_{S_G})|_{S_H}=(F^{-1}\circ G|_{S_G})|_{S_H}=F^{-1}\circ G\circ H^{-1}\circ H|_{S_H}.\end{equation*}
Note that the quantile function of $F^{-1}\circ G(Z)$ is $F^{-1}\circ G\circ H^{-1}$, hence the latter relation is equivalent to $F^{-1}\circ G(Z)\geq_\mathcal{C}^e Z$, for every random variable $Z$ with support included in that of $Y$. Given that the support of $\tilde{Y}_n$ is always included in the support of $Y$, this yields
$\tilde{X}_n=_dF^{-1}\circ G(\tilde{Y}_n)\geq^e_{\mathcal{C}}\tilde{Y}_n$. Now,  in case 1), as we are assuming that $T\in\mathcal{I}(\geq^e_{\mathcal{C}})$, it follows that $T(F^\ast_n)\geq T(G_n)$, while, in case 2), this holds by assumption. Note that we are not interested in comparing $G_n$ and $F^\ast_n$, however, this result contains information about the distributional behavior of $T(\F_n)$ and $T(\G_n)$. In fact, since $T(F^\ast_n)\geq T(G_n)$ holds for every possible pair of realizations $x_1^\ast,\ldots,x_n^\ast$ and $y_1,\ldots,y_n$, obtained as above, the latter relation can be equivalently expressed as $T(\F_n)\geq_{st} T(\G_n)$ (again, this characterization of the usual stochastic order is due to the probability integral transform, see Theorem~1.A.1 of \cite{shaked2007}).
\end{proof}

Theorem \ref{teo} demonstrates the practical usefulness of our Definition~\ref{general}. In fact, transform orders may sometimes not be meaningful, especially when choosing an unusual reference class $\mathcal{C}$. However, they may be employed merely as tools for deriving stochastic properties of statistical functionals. For instance, to apply part 1 of Theorem~\ref{teo}, it is enough to show that the functional of interest $T$ belongs to some isotonic family, $\mathcal{I}(\geq^e_{\mathcal{C}})$, and to do so, one may use existing results which ensure that $T\in\mathcal{I}(\geq^e_{\mathcal{C}})$. In particular, as we discuss in the next section, this approach works when $\mathcal{C}$ is the class of functions that defines the star order or the convex transform order, enabling the determination of the stochastic monotonicity of the main measures of inequality and skewness. Obviously, if $T$ is isotonic with more than one of the orders in Definition~\ref{general}, we should use the weakest one, in order to enlarge the range of applicability of the result. On the other hand, it might happen that our functional $T$ is not isotonic with known orders. In this case, we may try to define a new \textit{ad hoc} transform order such that $T\in\mathcal{I}(\geq^e_{\mathcal{C}})$, for instance by choosing a wider class $\mathcal{C}$ or by defining $\mathcal{C}$ according to the properties of $T$. Alternatively, we may apply part 2 of Theorem~\ref{teo}, replacing the isotonicity assumption, $T\in\mathcal{I}(\geq_\mathcal{C}^e)$, with  the weaker order-preserving property $T(F_n^*)\geq T(G_n)$.

Both approaches may be useful, depending on the situation. Several results available in literature assert the isotonicity of many popular probability functionals, with respect to some of the transform orders mentioned above, hence part 1 immediately provides the ordering between the corresponding statistical functionals, as discussed in the following section. When such isotonicity results are not available, part 2 offers a weaker order-preserving assumption, allowing for the same conclusion. This latter approach will be explored in Section~\ref{sec:tests}.

\section{Stochastic behavior of inequality measures}
\label{sec:3}
We now consider some important families of statistical functionals that are commonly employed in several fields, including statistics, economics and finance. It should be noted that functionals that are isotonic with the usual stochastic order are generally seen as \textit{location} functionals. Since every function of the random sample, which is nondecreasing in each argument, has a stochastically increasing behavior \citep[Theorem 1.A.3]{shaked2007}, then the stochastic monotonicity of \textit{location estimators} may be easily derived. Therefore, we focus on some other families of functionals, that are generally not isotonic with the usual stochastic order, namely inequality measures. As discussed later in this section, inequality measures must be isotonic with the star order, so they fit perfectly to our framework. Similarly, notice that isotonicity with the convex transform order is a basic condition for a skewness measure \citep{vanzwet1964,oja,eberl2021}, therefore our method may be also applied to skewness measures, but this is beyond the purpose of this paper.

Hereafter, we focus on non-negative random variables with finite mean. Let $L_F(p)=\int_0^pF^{-1}(t)\,dt/{\mu_X}$, $p\in[0,1]$ be the \textit{Lorenz curve} of $F$ (as usual, $\mu_X$ represents the mean of $X$). The Lorenz curve is a primary tool for representation of inequality (e.g. income inequality), as it is typically understood that the higher of two non-intersecting Lorenz curves shows less inequality. This gives rise to the definition of the \textit{Lorenz order}.
\begin{definition}
We say that $X$ dominates $Y$ with respect to the Lorenz order, denoted by $X\geq_{L}Y$, if $L_F(p)\leq L_G(p)$ for every $p\in[0,1]$.
\end{definition}
Note that the reverse relation ($L_F(p)\geq L_G(p)$) is sometimes used to define the Lorenz order. We use the same notation as in \cite{marshall2011}, that is, $X$ dominates $Y$ if it exhibits higher inequality, as measured by the Lorenz curve. Using standard arguments in the ordering theory, it is possible to derive several classes of probability functionals that are isotonic with $\geq_L$. In particular, any functional $I(X)$ satisfying the following basic properties (see \cite{shorrocks1980}), may be seen as an inequality measure:
\begin{enumerate}
\item $I\in \mathcal{I}(\geq_L)$;
\item $I(X)\geq I(\mu_X)$;
\item $I(aX)=I(X),\;a>0$;
\item $I(X+b)\leq I(X),\;b>0$.
\end{enumerate}
Actually, properties 2--4 are redundant, as they are implied by 1, that is, the Lorenz isotonicity of the inequality measure $I$ is the crucial property. It is well known that the star order implies the Lorenz order \citep{shaked2007}, so that $\mathcal{I}({\geq_{L}})\subset\mathcal{I}(\geq_*)$. However, since the star order is usually defined only in the continuous case, we need the following extension.
\begin{lemma}\label{inclusion}
$\mathcal{I}({\geq_{L}})\subset\mathcal{I}(\geq_*^e)$.
\end{lemma}
\begin{proof}
Choose $T\in\mathcal{I}({\geq_{L}})$. The inclusion follows if we prove that $F_n\geq_*^e G_n\implies F_n\geq_L G_n$, as this latter implies that $T(F_n)\geq T(G_n)$. Let $\tilde{F}_n$ and $\tilde{G}_n$ be the linear interpolators of the jump points of $F_n$ and $G_n$, respectively. If $F_n\geq_*^e G_n$, then we can take $\phi=\tilde{F}_n^{-1}\circ\tilde{G}_n$ in Definition~\ref{general}. The function $\phi$ coincides with the linear interpolator of ${F}_n^{-1}\circ{G}_n$, which is star-shaped by construction. Then, since $\tilde{F}_n$, $\tilde{G}_n$ are continuous, $\tilde{F}_n\geq_\ast\tilde{G}_n$ and the ratio ${\tilde{F}_n^{-1}}/{\tilde{G}_n^{-1}}$ is nondecreasing \citep[p. 214]{shaked2007}, the sequence ${{F}_n^{-1}( i/n)}/{{G}_n^{-1}(i/n)}$ is nondecreasing for $i=1,\ldots,n$, whereas ${{F}_n^{-1}(p)}/{{G}_n^{-1}(p)}$ is constant between ${(i-1)}/n$ and ${i}/n$, for $i=2,\ldots,n$. Hence, the function $R={{F}_n^{-1}}/{{G}_n^{-1}}$ is a nondecreasing step function. Without loss of generality, let $\mu_X=\mu_Y=1$. Then $L_{F_n}(p)-L_{G_n}(p)=\int_0^p({F}_n^{-1}(u)-{G}_n^{-1}(u))\,du$. If $R$ is nondecreasing, and the means are equal, the quantile functions must cross, and $R(\cdot)-1$ must have one sign change. However, $R(\cdot)-1$ has the same sign as ${F}_n^{-1}-{G}_n^{-1}$, therefore the argument of Theorem~4.B.4 of \cite{shaked2007} implies the result.
\end{proof}
Now, the following result is an immediate consequence of Theorem \ref{teo} and Lemma \ref{inclusion}.

\begin{corollary}\label{star}
If $I\in \mathcal{I}(\geq_L)$ and $X\geq_*Y$, then $I(\F_n)\geq_{st}I(\G_n)$.
\end{corollary}

In the next subsections, we will focus on some particularly relevant families of inequality measures.

\subsection{Expectation-type transformations}
It is well known that $X\geq_{L}Y$ if and only if $\E(\phi( X/{\mu_X}))\geq \E(\phi( Y/{\mu_Y}))$, for every convex function $\phi$ \citep{marshall2011}. Therefore, any functional of the form
\begin{equation*}
T_\phi(F)=\E\left(\phi\left(\tfrac X{\mu_X}\right)\right)=\int_0^\infty\phi\left(\tfrac x{\mu_X}\right)\,dF(x)= \int_0^1\phi\left(\tfrac{F^{-1}(p)}{\mu_X}\right)\,dp,
\end{equation*}
where $\phi$ is convex, is isotonic with $\geq_L$, as it is easily seen to satisfy the properties of inequality measures \citep{metron}. Now, the corresponding statistical functional is $T_\phi(\F_n)=\sum \phi({X_i}/{\overline{X}_n})/n$, where $X_1,\ldots,X_n$ is a random sample from $X$ and $\overline{X}_n$ is the sample mean. Several well-known indices belong to this general family, among which we may note the class of \textit{generalized entropy}, or \textit{additively decomposable} measures of inequality \citep{shorrocks1980,shorrocks}, obtained by setting $\phi_r(x)={x^r}/({r(r-1)})$, $r\neq0,1$, $\phi_0(x)=-\log x$, or $\phi_1(x)=x\log x$, respectively, which yield
\begin{equation*}
I_r(\F_n)=\begin{cases}\frac{1}{r(r-1)n} \sum (\frac{X_i}{\overline{X}_n})^r&r\neq0,1\\
\frac{1}{n} \sum \log(\frac{\overline{X}_n}{X_i})&r=0\\
\frac{1}{n} \sum \frac{X_i}{\overline{X}_n}\log(\frac{X_i}{\overline{X}_n})&r=1.
\end{cases}
\end{equation*}
This class gives, for $r=1$, the Theil index, that is, a shifted version of the Shannon's entropy measure, applied to the quantities ${X_i}/({n\overline{X}_n})$ instead of probabilities; for $r\in(0,1]$, a monotonic transformation of the Atkinson's class \citep{atkinson1970}; and, for $r=2$, a simple transformation of the coefficient of variation $CV_n$, that is, $(2I_2)^{1/2}=CV_n$. Moreover, when $\phi(x)=|x-1|$, we obtain the relative mean absolute deviation $T_\phi(\F_n)=\sum |X_i-\overline{X}_n|/({n\overline{X}_n})$. Then, the following result is an immediate consequence of Corollary \ref{star}.
\begin{corollary}\label{cor1}
If $X\geq_*Y$, then $\sum \phi({X_i}/{\overline{X}_n})/n\geq_{st}\sum \phi({Y_i}/{\overline{Y}_n})/n$, for every convex function $\phi$.
\end{corollary}
Notice that \cite{greenwood} proved that the Greenwood statistic has a stochastic increasing behavior with respect to the star order. However, since the Greenwood statistic is a transformation of the coefficient of variation, namely ${(1+CV_n^2)}/n$, Theorem~1 of \cite{greenwood} follows as a consequence of Corollary \ref{cor1}.

\subsection{Distorted expectations}
Let $H$ be a \textit{distortion function}, that is a non-decreasing function on the unit interval, such that $H(0)=0$ and $H(1)=1$, and let $\tilde{H}(p)=1-H(1-p)$ be the corresponding dual distortion function. Probability functionals of the form
\begin{equation*}
\E_H(F)=\frac1{\mu_X}\int_0^\infty x\,dH\circ F(x)=\frac1{\mu_X}\int_0^1 F^{-1}(p)\,dH(p)=\frac1{\mu_X}\int_0^\infty\tilde{H}(1-F(x))\,dx,
\end{equation*}
are generally referred to as distorted expectations, distortion risk measures, or Gini-type functionals \citep{wang1998,muliere1989,landodist}. It can be seen that $X\geq_{L}Y$ if and only if $\E_H(F)\leq (\geq) \E_H(G)$, for every concave (convex) distortion $H$. Then, $\E_H$ is antitonic with respect to $\geq_L$, provided that $H$ is concave. If we denote by $X_{(1)},\ldots,X_{(n)}$ the order statistics of a random sample from $X$, the corresponding statistical functional is
\begin{equation*}
\E_H(\F_n)=\frac1{\overline{X}_n}\sum{X_{(i)}}\int^{\frac in}_{\frac{i-1}n}\,dH(p)
=\frac1{{\overline{X}_n}}\sum{X_{(i)}}\left[H\left(\tfrac in\right)-H\left(\tfrac{i-1}n\right)\right].
\end{equation*}
Such linear combinations of order statistics are generally referred to as \textit{L-statistics} \citep{serfling}. Several important indices belong to this family. For instance, by choosing $H(p)=1-(1-p)^k$, $k\geq1$, we obtain the class of generalized Gini indices $\Gamma_H=1-\E_H$, introduced by \cite{donaldson}. In particular, the classic Gini coefficient of inequality \citep{gini} is given by $\Gamma=\Gamma_H$, with ${H}(p)=1-(1-p)^2$, that is
\begin{equation*}
\Gamma(\F_n)=1-\frac1{n^2{\overline{X}_n}}\sum{X_{(i)}}(2n-2i+1)=1-2\int_0^1L_{\F_n}(p)\,dp=\frac{\sum_{i=1}^n\sum_{j=1}^n|X_{(i)}-X_{(j)}|}{2n^2{\overline{X}_n}}.
\end{equation*}
As for the expressions above, we recall that there are several alternative ways to represent $\Gamma$ (see Chapter 1 of \cite{shlomo}).
For example, note that, if $k$ is a positive integer, $\int_0^\infty(1-F(x))^k\,dx=\E(\min(X_1,\ldots,X_k))$, so that $\Gamma(F)=1-{\E(\min(X_1,X_2))}/{\mu_X}$.

Similarly, the functional $\int_0^\infty F^{-1}(p)h(p)\,dp/{\mu_X}$, where $h$ is non-decreasing and such that $H(p)=\int_0^ph(t)\,dt$, $p\in[0,1]$, is isotonic with respect to $\geq_L$. If we set $w(p)=h(p)-1$, $p\in[0,1]$, without loss of generality in terms of isotonicity, we obtain the family of \textit{linear inequality measures}
\begin{equation*}
\tilde{\Gamma}_w(F)=\frac 1{\mu_X}\int_0^1 F^{-1}(p)w(p)\,dp,
\end{equation*}
studied by \cite{mehran}. In particular, it can be shown that the weight function $w(p)=2p-1$ yields again the Gini coefficient, whereas $w(p)=0$, $p\in(0,1)$, $w(0)=-1$, $w(1)=1$ gives the relative range $\tilde{\Gamma}_w(\F_n)= (X_{(n)}-X_{(1)})/{\overline{X}_n}$. It is easy to see that $\Gamma_H$ and $\tilde{\Gamma}_w$, as probability functionals, satisfy the four properties of inequality measures discussed earlier. Again, Corollary~\ref{star} gives the following {result}.
\begin{corollary}\label{cor2}
If $X\geq_*Y$, then $\Gamma_H(\F_n)\geq_{st} \Gamma_H(\G_n)$, for every concave distortion function $H$, and $\tilde{\Gamma}_w(\F_n){\geq_{st}} \tilde{\Gamma}_w(\G_n)$, for every non-decreasing weight function $w$ on $[0,1]$ such that $\int_0^1w(p)\,dp=0$.
\end{corollary}

\subsection{Some useful remarks}
\label{sec:useful}
One may wonder if, in Theorem \ref{teo}, a transform order may be replaced by a different kind of order $\succ$, that is, if, in general, $F\succ G$ and $T\in\mathcal{I}(\succ)$ imply $T(\F_n)\geq_{st}T(\G_n)$, where $\succ$ is not a transform order. For instance, can one replace $\geq_*$ by $\geq_L$ in Corollary \ref{star}? More generally, one may even wonder if the stochastic ordering assumption between the baseline distributions could be relaxed: for instance, would $T(F)\geq T(G)$ suffices to obtain $T(\F_n)\geq_{st}T(\G_n)$? (This would also mean that $T(F)= T(G)$ implies $T(\F_n)=_{d}T(\G_n)$). The following counterexample provides a negative answer to these conjectures.

Let $X$ and $Y$ be discrete random variables with uniform distributions on the supports $\{1, 3.5, 6, 6.5, 9, 11\}$ and $\{2, 3, 5, 7, 7.5, 10\}$, respectively.
It is easily seen that $F\geq_L G$ and {the Gini indices are such that} $\Gamma(F)>\Gamma(G)$, but $F\not\geq_*^e G$. Note that we focus on discrete distributions, because, using continuous parametric models, it is difficult to find instances such that $F\geq_L G$ and $F\not\geq_*^e G$. In the discrete case, the distributions of $\Gamma(\F_n)$ and $\Gamma(\G_n)$ are also discrete, with finite support included in the unit interval. Therefore, we may obtain a precise approximation of these distributions, for small sample sizes, by using a large number of simulation runs. An approximate representation of the CDFs of $\Gamma(\F_n)$ and $\Gamma(\G_n)$, based on one million random samples of size $n = 3$ from $F$ and $G$, is shown in Figure~\ref{f1}. The functions are clearly crossing, hence $\Gamma(\F_n)\not\geq_{st} \Gamma(\G_n)$. In particular, we observe that these functions exhibit some ``bumps'' within different intervals, and this represents an obstruction for the dominance relation. When the distributions are star ordered ($F\geq_\ast^e G$), such bumps occur within some overlapping intervals, indeed the dominance relation is guaranteed by Corollary~\ref{star}. We obtained a similar behavior for other small sample sizes, such as $n=4,5$. This result is quite unexpected, however, we may conclude that both conjectures above are false. In particular, the behavior of the functionals is not ``under control'' if the sample size is small, whereas we know that, for large sample sizes, under some conditions (which are satisfied for probability functionals of the form $\int t(x)\,dF(x)$, such as transformed and distorted expectations), statistical functionals converge to the constant value $T(F)$. Therefore, transform orders provide suitable tools for controlling the stochastic behavior of statistical functionals.

\begin{figure}
\centering
\includegraphics{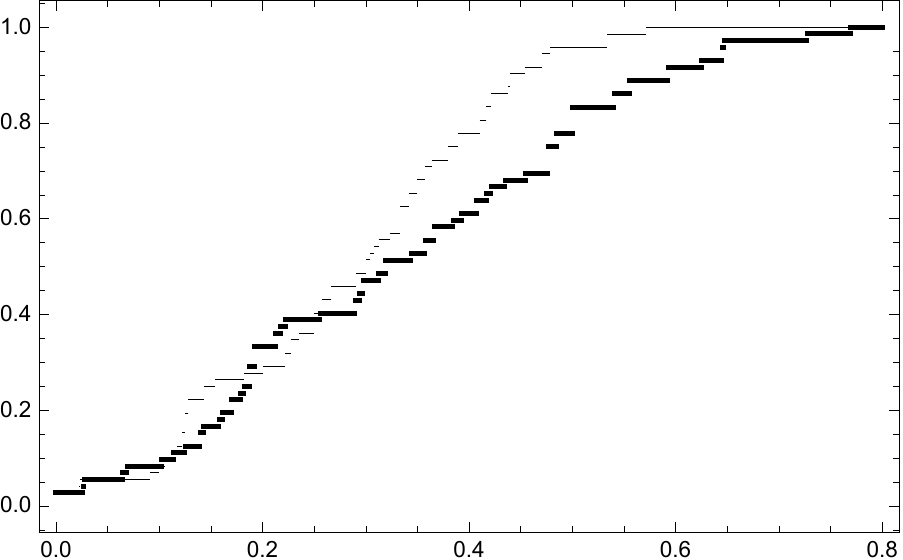}
\caption{CDFs of the Gini coefficient, generated by {$10^6$} samples from $F$ (thick) and $G$ (thin).\label{f1}}
\end{figure}

\section{Application to goodness-of-fit tests}
\label{sec:tests}
In the statistical literature, tests of the null hypothesis $\mathcal{H}^R_0:G^{-1}\circ F|_{R}\in\mathcal{C}$, versus the alternative $\mathcal{H}^R_1:G^{-1}\circ F|_{R}\notin\mathcal{C}$, where $R\subset S_F$ is some interval, $F$ is the distribution of interest and $G$ is known, may be particularly interesting. Hereafter, we will focus on the case in which $G(0)=0$, to avoid some technical issues. Given $G$, $G^{-1}\circ F$ is called the \textit{generalized hazard function} of $F$ \citep{barlow1971}. Tests for $G\geq_c F$ and $G\geq_* F$ have several applications: for example, by choosing $G(x)=\mathcal{E}(x):=1-\exp(-x)$, $x\geq0$, namely, the unit exponential distribution, we obtain tests for the \textit{increasing hazard rate} (IHR) and the \textit{increasing hazard rate average} (IHRA) properties, respectively, which are fundamental tools for survival analysis and reliability theory \citep{marshall2007}. {Non parametric tests of these properties have been studied extensively in the literature, most of which deal with the null hypothesis of exponentiality ($F=\mathcal{E}$) versus the alternative that $F$ is IHR \citep{proschan1967,barlow1969,bickel1969,bickeldoksum} or IHRA \citep{deshpande,klefsjo,kochar1985testing,link,wells,ahmad}, but not exponential. In such cases, it is clear that the distribution of the test statistic can be obtained using $F=\mathcal{E}$. Differently, we study goodness-of-fit tests where the determination of the distribution of the test statistic under the (composite) null hypothesis is not obvious and it represents an important issue \citep{tenga1984,hall2005,durot2007,groeneboom2012,gibels,lando2021ssd,landotest}. As shown in the sequel, part 2 of Theorem \ref{teo} may be used to establish stochastic ordering properties of test statistics under some circumstances of interest. These results make it possible to control the behavior of the test for a given sample size, which is an important aspect, comparable to asymptotic properties.}

{In the next subsections, we shall remove the restriction $R$ and just consider $\mathcal{H}_0:G^{-1}\circ F\in\mathcal{C}$ {versus $\mathcal{H}_1:G^{-1}\circ F\notin\mathcal{C}$}, to simplify the notations. It is easily seen that the results in Subsections~\ref{subsec4.1} and \ref{subsec4.2} hold even in the restricted case. Differently, a restricted null hypothesis is considered in Subsection~\ref{subsub:IHRA}.}

\subsection{Testing convexity}
\label{subsec4.1}
Define the \textit{greatest convex minorant} of a function $\phi$ as the largest convex function $\phi_c$ that does not exceed $\phi$ (see \cite{tenga1984} for {a formal definition} of this function). An intuitive way of testing $\mathcal{H}_0:G\geq_c F$ {versus $\mathcal{H}_1:G\not\geq_c F$ }consists in measuring the distance between $G^{-1} \!\circ\! \F_n$ and its greatest convex minorant $(G^{-1} \!\circ\! \F_n)_{c}$, {or, more generally,}
\begin{multline*}
T^{{K}}_c(\F_n)=\sup_{i=2,\ldots,n-1} \left|K\circ G^{-1} \!\circ\! \F_n(X_{(i-1)})-K\circ(G^{-1} \!\circ\! \F_n)_{c}(X_{(i)})\right|\\
=\sup_{i=2,\ldots,n-1} \left(K\circ G^{-1} (\tfrac{i-1}n)-K\circ(G^{-1} \!\circ\! \F_n)_{c}(X_{(i)})\right),
\end{multline*}
where {$K$ is some strictly increasing transformation defined on $S_G$. This family of tests is quite general. Some interesting examples are obtained by choosing $K$ as the identity function \citep{tenga1984,durot2007,lando2021ssd}; or setting $K=G$ \citep{groeneboom2012,landotest}, in this latter case $T^K_c(\F_n)$ boils down to a Kolmogorov-Smirnov type test statistic. It can be seen that $T_c^K$ is scale invariant, which is a basic condition for tests of this kind. Note that the difference is evaluated among the shifted points $X_{(i)}$ and $X_{(i-1)}$ because, by construction, $G^{-1} \!\circ\! \F_n(X_{(i-1)})\geq (G^{-1} \!\circ\! \F_n)_{c}(X_{(i)})$}. A weighted version of $T^K_c(\F_n)$ may also be considered without loss of generality. Of course, since $\mathcal{H}_0$ is a composite hypothesis, the determination of the \textit{least favorable distribution} of $T^K_c(\F_n)$ under $\mathcal{H}_0$ is especially critical. With regard to the IHR hypothesis, this distribution is indeed obtained by simulating from the unit exponential, as proved by \cite{tenga1984}. We remark that a more general result may be achieved as an application of Theorem~\ref{teo}. As expected, the least favorable distribution of $T^K_c(\F_n)$ under $\mathcal{H}_0:G\geq_c F$, is determined just by simulating from $G$. Let $\Phi^G_p(F)=(G^{-1}\circ F)_c(F^{-1}(p))$. The following lemma is a consequence of Theorem 2.2 of \cite{tenga1984}, and it establishes the order-reversing behavior of $\Phi^G_p$.

\begin{lemma}\label{lemma}
{Assume that} $G\geq_c F$. Then $\Phi^G_{i/n}(G_n)\leq \Phi^G_{i/n}(F^*_n)$, for $i=1,\ldots,n$.
\end{lemma}
Theorem \ref{teo} and Lemma \ref{lemma} imply the following result. Note that this is a slight generalization of existing results \citep{tenga1984,lando2021ssd}, however, here we show it to illustrate our method, as an immediate application of part 2 of Theorem~\ref{teo}.

\begin{proposition}\label{test}
Under $\mathcal{H}_0:G\geq_c F$, $T^K_c(\G_n)\geq_{st}T^K_c(\F_n)$.
\end{proposition}
\begin{proof}Clearly, by Lemma \ref{lemma}, {$\Phi^G_{ i/n}(G_n)\leq \Phi^G_{ i/n}(F^*_n)$}. Note that
$$T^K_c(\F_n)=\sup_{i=2,\ldots,n-1} \left({K\circ}G^{-1}(\tfrac {i-1}n)-{K\circ}\Phi^G_{\frac in}(\F_n)\right),$$ then Theorem \ref{teo} implies the result{, since $K$ is increasing}.\end{proof}
Therefore, we reject $\mathcal{H}_0$ when $T^K_c(F_n) \geq c_{\alpha,n}$, where $c_{\alpha,n}$ is the solution of $P(T^K_c(\G_n)\geq c_{\alpha,n}) = \alpha${. Proposition \ref{test} ensures that} the \textit{size} of the test { is $\alpha$, that is, under $\mathcal{H}_0$, $P(T^K_c(\F_n)\geq  c_{\alpha,n})\leq \alpha$}. {Moreover, under the alternative ``$G^{-1}\circ F$ is concave'', reversing Proposition \ref{test}, we can similarly obtain $P(T^K_c(\F_n)\geq c_{\alpha,n})\geq P(T^K_c(\G_n)\geq c_{\alpha,n})=\alpha$, that is, the test is \textit{unbiased} under concave alternatives.} The $p$-value of the test is $p=P(T^K_c(\G_n)\geq T^K_c(F_n))$.

As well known, the determination of the least favorable distribution of $T^K_c$ yields conservative tests. In order to improve the power of the test against non-convex alternatives ($\mathcal{H}_1:G\not\geq_c F$), a more modern approach consists in performing bootstrap resampling from the distribution $F^G_n=G\circ(G^{-1}\circ F_n)_c$, that is an isotonic estimate of $F$, instead of sampling from $G$ \citep{groeneboom2012}. In particular, denote with $\hat{\F}^G_{n}$ the (bootstrap) empirical distribution obtained by sampling from $F^G_n$. The following result holds.
\begin{proposition}\label{boot}
Under $\mathcal{H}_0:G\geq_c F$, $T^K_c(\G_n)\geq_{st} T^K_c(\hat{\F}^G_{n})$.
\end{proposition}
\begin{proof}
Let $F^G_n$ be the estimate of $F$ which corresponds to the empirical distribution $F_n^*$. $F_n^G$ is a realization of $\hat{\F}^G_{n}$. Lemma \ref{lemma} yields $\Phi^G_{ i/n}(G_n)\leq \Phi^G_{ i/n}(F^*_n)$, for $i=1,\ldots,n$. Note that $F^G_n$ is a minorant of $F^*_n$ (equivalently, $(F^*_n)^{-1}$ is a minorant of $(F^G_n)^{-1}$). Then, since $G^{-1}\circ F_n^G=(G^{-1}\circ F^*_n)_c$, we obtain
\begin{multline*}\Phi^G_{\frac in}(G_n)\leq\Phi^G_{\frac in}(F^*_n)=(G^{-1}\circ F_n^*)_c((F^*_n)^{-1}(\tfrac in))\\ \leq(G^{-1}\circ F_n^*)_c((F^G_n)^{-1}(\tfrac in))= G^{-1}\circ F_n^G\circ(F^G_n)^{-1}(\tfrac in)=\Phi^G_{\frac in}(F^G_n).\end{multline*}
Hence, as shown in the proof of Proposition \ref{test}, we obtain that $ T^K_c(\G_n)\geq_{st} T^K_c(\hat{\F}^G_{n})$.
\end{proof}
Since quantiles are isotonic with $\geq_{st}$, Proposition \ref{boot} ensures that, with this method, we obtain a smaller critical value and, correspondingly, a higher value of the power function under $\mathcal{H}_1$. {On the other hand, a test obtained by this method does not satisfy anymore the properties implied by Proposition \ref{test} (size and unbiasedness).}

\subsection{Testing star-shapedness}
\label{subsec4.2}
{Now, let us consider tests for $\mathcal{H}_0:G\geq_* F$ versus $\mathcal{H}_1:G\not\geq_* F$.} Note that $G\geq_* F$ requires that $F(0)=0$. Define the \textit{greatest star-shaped minorant} of a function $\phi$, as the largest star-shaped function $\phi_*$ that does not exceed $\phi$ {(see Sect. 2 of \cite{Wang1987} for a formal definition)}.
The same arguments used earlier may determine the least favorable distribution, {under $\mathcal{H}_0$,} of the test statistic
\begin{equation*}\label{starrr}
{T^K_*(\F_n)
=\sup_{i=2,\ldots,n-1} \left(K\circ G^{-1} (\tfrac{i-1}n)-K\circ(G^{-1} \!\circ\! \F_n)_{*}(X_{(i)})\right).}
\end{equation*}
Let $\Psi^G_p(F)=(G^{-1}\circ F)_*(F^{-1}(p))$. The following lemma establishes the order-reversing behavior of $\Psi_p^G$, corresponding to Lemma~\ref{lemma}.

\begin{lemma}\label{lemma2}
Assume that $G\geq_\ast F$. Then $\Psi^G_{ i/n}(G_n)\leq \Psi^G_{ i/n}(F^*_n)$, for $i=1,\ldots,n$.
\end{lemma}
\begin{proof}
To simplify notations, let $y_1,\ldots,y_n$ be ordered realizations {from $G$}. The function $u=F^{-1}\circ G$ is, by assumption, anti-star-shaped, and recall that $x_i^*=u(y_i)$ are the observations which determine $F_n^*$. Note that $G^{-1}\circ G_n(y_i)=G^{-1}(i/n)$. We now describe explicitly $(G^{-1}\circ G_n)_*$ (see \cite{wang1988}). Let
$$a_i=\frac{1}{y_i}G^{-1}\left(\frac{i-1}n\right),\quad i=1,\ldots,n,$$
be the slope of the line connecting the origin to the point $(y_i,G^{-1}(({i-1})/n))$. The analytical expression of the greatest star-shaped minorant is given by $(G^{-1}\circ G_n)_*(x)=\alpha_ix$, for $x\in[y_{i-1},y_i)$, where {$y_0=0$} and $\alpha_i=\min\{a_j,\,j=i,\ldots,n\}$ (note that $\alpha_1=0$, and, moreover, star-shapedness requires that the slopes are increasing). $(G^{-1}\circ F^*_n)_*$ is defined similarly. Note that the two functions have different domains, but coincide at corresponding points: $G^{-1}\circ F^*_n(x_i^*)=G^{-1}\circ G_n(y_i)=G^{-1}(({i-1})/n)$. We want to prove that $(G^{-1}\circ F^*_n)_*(u(y_i))\geq (G^{-1}\circ G_n)_*(y_i)$, for $i=1,\ldots n$. It is sufficient to prove
\begin{equation*}
\frac{G^{-1}(\frac {k-1}n)}{u(y_k)}u(y_i)\geq \frac{G^{-1}(\frac {k-1}n)}{y_k}y_i \iff \frac{u(y_i)}{y_i}\geq \frac{u(y_k)}{y_k},
\end{equation*}
for $k=i,\ldots,n$, which follows form the anti-star-shapedness of the function $u$, that is, $u(x)/x$ is decreasing.
\end{proof}
{Using Lemma \ref{lemma2} and similar arguments as those used in the previous subsection, we can ensure that the least favorable distribution of $T^K_*$, under $\mathcal{H}_0$, is determined by simulating from $G$. Accordingly, the test determined by $P(T^K_*(\G_n)\geq c_{\alpha,n}) = \alpha$ has size $\alpha$ and it is unbiased against anti-star-shaped alternatives. Moreover, the bootstrap estimator which arises from the isotonic estimate $G\circ(G^{-1}\circ F_n)_*$, denoted as $\tilde{\F}^G_{n}$, satisfies an equivalent property to that established by Proposition \ref{boot}. We summarize these results as follows (the proof is omitted).
\begin{proposition}\label{star teo}Under $\mathcal{H}_0:G\geq_* F$,
\begin{enumerate}\item  $T^K_*(\G_n)\geq_{st}T^K_*(\F_n)$;
\item $T^K_*(\G_n)\geq_{st} T^K_*(\tilde{\F}^G_{n})$.
\end{enumerate}
\end{proposition}
Differently from the tests discussed in the Section 4.1, which have been studied extensively in the literature, tests of the null hypothesis $G\geq_* F$ {are relatively} unexplored, so it is worth discussing them further. In fact, our results can be useful for testing goodness-of-fit to the IHRA family, which is particularly relevant in reliability and survival analysis, or some generalizations of this property, as described by \cite{barlow1971}.}

{\subsubsection{An example: testing the IHRA property}}
\label{subsub:IHRA}
{A CDF $F$ is said to be IHRA if the \textit{hazard function} $H=-\ln(1-F)$ is star-shaped, or equivalently, if the \textit{hazard rate average} (HRA) function ${H(x)}/x$ is increasing \citep{marshall2007}. Tests of the IHRA property have already been considered in the literature, however, these are based on the null hypothesis that $F$ is exponential, versus IHRA alternatives \citep{deshpande,klefsjo,kochar1985testing,link,wells,ahmad}, which is clearly different from our approach.}

Plugging {$G=K=\mathcal{E}$} into the expression of $T^K_*$, we obtain a Kolmogorov-Smirnov type test statistic, which can be used to test goodness-of-fit to the IHRA family. More generally, we may consider the following restricted hypotheses $\mathcal{H}^\nu_0: H|_{S_\nu}$ is star-shaped versus $\mathcal{H}^\nu_1: H|_{S_\nu}$ is not star-shaped, where $S_\nu=\{x:x\leq F^{-1}(1-\nu)\}$, $\nu\in[0,1]$ (clearly $\mathcal{H}_0$ coincides with $\mathcal{H}_0^\nu$). Let us define $n_\nu$ by $X_{(n_\nu )} = \F_n^{-1}(1-\nu)$, and $S_{n,\nu}=\{x\leq X_{(n_\nu )}\}$. Accordingly, we have the following restricted test statistic:
\begin{equation*}
T^\nu_*(\F_n)=\sup_{i=2,\ldots,n_\nu-1} \left(\tfrac{i-1}n-\mathcal{E}\circ\left(\mathbb{H}_n|_{S_{n,\nu}}\right)_{*}(X_{(i)})\right),
\end{equation*}
where $\mathbb{H}_n=-\ln(1-\F_n)$ is the empirical hazard function.
It is easy to see that this restricted test still satisfies the properties established by Proposition \ref{star teo}, for every $\nu\in[0,1]$. Therefore, the distribution of $T^\nu_*$ is obtained by simulating from $G=\mathcal{E}$. We reject $\mathcal{H}_0$ when $T_*^\nu(F_n) \geq c_{\alpha,n,\nu}$, where $c_{\alpha,n,\nu}$ is the solution of $P(T_*^\nu(\G_n)\geq c_{\alpha,n,\nu})= \alpha$, giving the size of the test. The test is unbiased against alternatives that exhibit a decreasing hazard rate average (DHRA), corresponding to an anti-star-shaped behavior of $H$, in $S_\nu$. Moreover, for every $\nu\in(0,1]$, it is possible to establish the following consistency property.
\begin{theorem}
If $\mathcal{H}_0^\nu$ is false, $\lim_{n\rightarrow \infty}P(T_*^{\nu}(\F_n)\geq c_{\alpha,n,\nu})=1$, for every $\nu>0$.
\end{theorem}
\begin{proof}
{Note that $T^\nu_*(\F_n)=\sup_{S_{n,\nu}}\left|\mathbb{F}_n(x)-\mathcal{E}\circ(\mathbb{H}_n|_{S_{n,\nu}})_*(x)\right|-1/n$, so we may equivalently consider the ``$\sup$'' on the right-hand side, which is easier to manipulate, instead of $T^\nu_*$. By the Glivenko-Cantelli Theorem, $\F_n$ converges a.s. and uniformly to $F$, whereas $\mathcal{E}^{-1}(p)=-\ln(1-p)$ is uniformly continuous in $[0,1-\nu]$ (for $\nu>0$), then $\mathbb{H}_n=\mathcal{E}^{-1}\circ \F_n$ converges a.s. and uniformly to $H=\mathcal{E}^{-1}\!\circ\! F$ in $S_\nu$. If $\mathcal{H}_{0}^\nu$ is true, we can apply Lemma 3.1 of \cite{wang1988}, which yields
\begin{equation*}
\sup_{S_{n,\nu}}\left|\mathbb{H}_n(x)-H(x)\right|\geq\sup_{S_{n,\nu}}\left|(\mathbb{H}_n|_{S_{n,\nu}})_*(x)-H(x)\right|,
\end{equation*}
implying strong uniform consistency of $(\mathbb{H}_n|_{S_{n,\nu}})_*$ in ${S_\nu}$ (clearly, if $\nu>0$, $\F_n^{-1}(1-\nu)$ converges a.s. to $F^{-1}(1-\nu)$). Since $\mathcal{E}$ is absolutely continuous, then $\F^\mathcal{E}_{n,\nu}:=\mathcal{E}\circ (\mathbb{H}_n|_{S_{n,\nu}})_*$ converges a.s. and uniformly to $F$ in $S_\nu$, for every $\nu>0$. Now, it can be easily seen that, under $\mathcal{H}_{0}^\nu$, $c_{\alpha,n,\nu}\rightarrow0$ a.s. for $n\rightarrow\infty$.}

{Define $F^\mathcal{E}_\nu=\mathcal{E}\circ (\mathcal{E}^{-1}\circ F|_{S_\nu})_*$. Suppose that $\mathcal{H}_{0}^\nu$ is false, that is, $\mathcal{E}\circ F$ is not star-shaped in $S_\nu$, and let $d=\sup_{S_\nu} |F-F^\mathcal{E}_\nu|=\sup_{S_\nu} (F-F^\mathcal{E}_\nu)$. Then, $F_n$ converges uniformly to $F$, whereas $\F^{\mathcal{E}}_{n,\nu}$ converges uniformly to $F^\mathcal{E}_\nu$ on $S_\nu$, with probability 1. Moreover, by assumption, $d>0$. Therefore, given some $\epsilon\in(0,d)$, there exists some $n_0$ such that, for $n>n_0$, $\sup_{S_{n,\nu}} |F_n-F|<\epsilon/2$ and $\sup_{S_{n,\nu}} |\F^{\mathcal{E}}_{n,\nu}-F^\mathcal{E}_\nu|<\epsilon/2$, with probability 1. Then, for $n>n_0$,\begin{equation*}
\F_n(x)-\F^{\mathcal{E}}_{n,\nu}(x)> F(x)-\frac\epsilon2-(F^\mathcal{E}_\nu(x)+\frac\epsilon2)=F(x)-F^\mathcal{E}_\nu(x)-\epsilon
\end{equation*}
almost surely, for every $x\in S_{n,\nu}$, which implies
\begin{equation*}
\sup_{S_{n,\nu}}\left|\F_n(x)-\F^{\mathcal{E}}_{n,\nu}(x)\right|> \sup_{S_{n,\nu}}\left| F(x)-\F^{\mathcal{E}}_{n,\nu}(x)-\epsilon\right|=\sup_{S_{n,\nu}} \left(F(x)-\F^{\mathcal{E}}_{n,\nu}(x)\right)-\epsilon=d-\epsilon>0.
\end{equation*}
Therefore, since $\epsilon$ can be arbitrarily small, $P(\sup_{S_{n,\nu}}\left|\F_n(x)-\F^{\mathcal{E}}_{n,\nu}(x)\right|\geq d)\rightarrow 1$, for $n\rightarrow\infty$. But since $c_{\alpha,n,\nu}\rightarrow0$, then $P(T_*^\nu(\F_n)\geq c_{\alpha,n,\nu})\rightarrow1$.}\end{proof}

{The value of $\nu$ should be chosen as small as possible, e.g. $\nu=0.01,0.05$. One might even choose $\nu$ as a function of $n$, for example $\nu=n^{-1/2}$. Yet, smaller $\nu$'s yield slower convergence, as we will need more observations in order to get $c_{\alpha,n,\nu}$ close enough to 0. In the limit case $\nu=0$, we cannot ensure uniform a.s. consistency of $\F^\mathcal{E}_{n,0}$ in $S_0$, accordingly, the consistency of the test $T^0_*$ cannot be established.}

{\subsubsection{Simulations}}

{The {numerical performance} of the test $T^\nu_*$ can be depicted by a simulation study. We focus on some critical alternatives, obtained using popular parametric families. By scale invariance, the scale parameters are always set to 1. We consider DHRA, bathtub-shaped (decreasing-increasing) HRA, and bell-shaped (increasing-decreasing) HRA models. The Weibull distribution with CDF $F(x)=1-e^{-x^a}$, $a,x>0$, is DHRA for $a<1$ and IHRA for $a\geq1$; the beta distribution, with CDF $F(x)=\int_0^x t^{a-1} (1-t)^{b-1}\,dt/{B(a,b)}$, $a,b>0$, $x\in[0,1]$, is bathtub-shaped HRA for $a<1<b$ and IHRA for $a,b\geq1$; the Burr distribution with CDF $F(x)={1}-\left({1+x^{b}}\right)^{-a}$, $x,a,b>0$, is DHRA for $b\leq1$ and bell-shaped HRA for $b>1$.}

The results {are reported in Figure 2, which contains the plots of the rejection probabilities at level $\alpha=0.1$, estimated using 2000 simulation runs. To cover a large number of alternatives, for each distribution considered we fixed a grid of parameters} within some intervals of interest. We considered sample sizes $n=10,50,100,200,300$, setting $\nu=0.05$, so the restriction has an effect just for $n\geq 50$. The {figures show} that the test has a remarkable performance in terms of empirical power, coherently with the behavior of the HRA. In fact, some alternatives considered are particularly difficult to detect, as the HRA may exhibit a very slight decrease (Weibull with $a\in[0.8,1)$) or a subtle non-monotonic behavior (beta with $a\geq0.7$; Burr with $b\geq1.5$). In the first case, the power is always larger than $\alpha$, independently of the sample size, confirming the unbiasedness property implied by Proposition \ref{star teo}; in the latter case, the power may be smaller than $\alpha$, for small values of $n$ (beta with $a\geq0.7$), confirming that unbiasedness of tests of this type holds just against DHRA alternatives, but when the HRA has a non-monotonic behavior, this is not guaranteed. Under the null hypothesis, the size is always bounded by $\alpha$, even when the HRA is close to the boundary of DHRA (Weibull with $a> 1)$. Clearly, as the test is consistent, the empirical power increases with sample size, when $\mathcal{H}_0^{0.05}$ is false{, even under the most citical alternatives} (this should be evaluated just for $n\geq50$, as remarked earlier; the beta with {$a\geq0.7$ is the most critical case and the power may be increased by increasing the sample size, or by choosing a larger value of $\nu$}).
\begin{figure}[h!]
\centering
\includegraphics[viewport=88 225 528 735,clip=true]{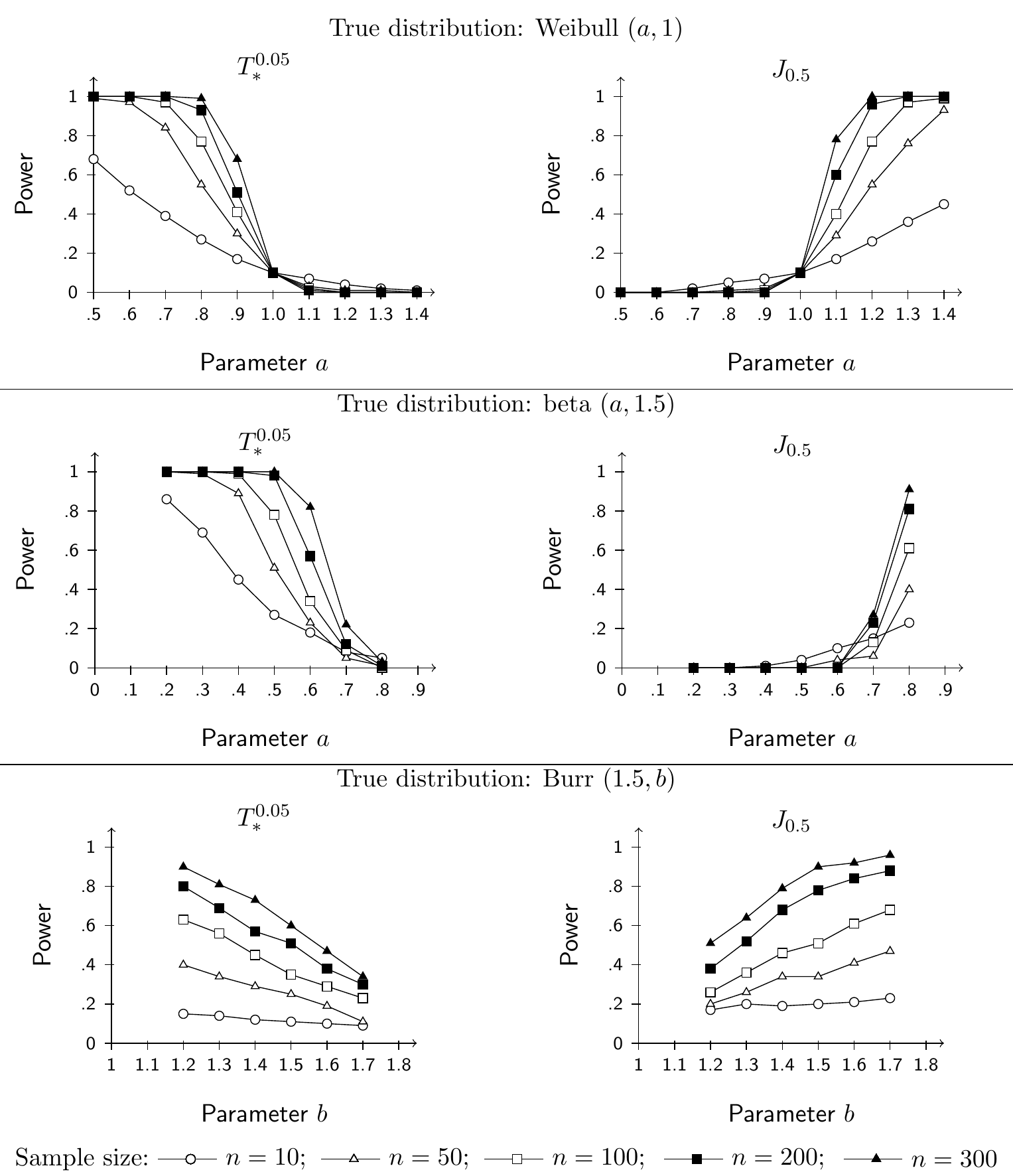}
\caption{Rejection rates of $T_*^{0.05}$ and $J_{0.5}$ {for} IHRA (Weibull $a\geq 1$), DHRA (Weibull $a< 1$), bathtub-shaped HRA (beta) and bell-shaped HRA (Burr) {distributions}.}
\end{figure}

The test is compared with the one proposed by \cite{deshpande} and later studied by \cite{bandy}, aimed at testing the null hypothesis $\widetilde{\mathcal{H}}_0:F=\mathcal{E}$ versus $\widetilde{\mathcal{H}}_1$: ``$F$ is IHRA and not exponential'', as in \cite{klefsjo}, \cite{kochar1985testing}, \cite{link}, \cite{wells}, \cite{ahmad}. This approach is quite different from ours, in fact, such tests are supposed to reject the IHRA hypothesis more easily, when it is false, compared to $T_*^\nu$. However, our simulation results reveal some possible issues of these kinds of tests. Actually, \cite{deshpande} introduced a family of rank-based test statistics, denoted as $J_p$, where the parameter $p$ ranges in $(0,1)$. As proved by \cite{bandy}, the choice for $p$ does not have consequences on the properties of the test, therefore, in the plots we show the results of $J_{0.5}$, which is the test recommended by \cite{deshpande}. This test is unbiased and consistent, for every $p$, under IHRA alternatives. However, {deviations from exponentiality should not necessarily imply that $\widetilde{\mathcal{H}}_1$ is true, as }tests of this type are not intended to deal with bathtub or bell-shaped HRA models; indeed $J_p$ may be misleading when $F$ has a non-monotone HRA, as it can be seen in the case of the beta distribution ($a\geq0.7$) or Burr distribution. In such cases, by increasing the sample size we also increase the power, that is, {the probability of rejecting $\widetilde{\mathcal{H}}_0$, suggesting that the (false) strict IHRA alternative may be true} (we obtained similar results for several values of $p$). {In other words, in some critical bathtub and bell-shaped HRA situations, $J_{p}$ might lead to the wrong decision. }So, given that non-monotone HRA distributions may be quite common, we recommend using $T_*^\nu$. Looking at {the plots}, it is important to remember that $T_*^\nu$ and $J_p$ cannot be compared in terms of power, since they deal with reversed hypotheses, {hence} one may only compare their overall performance. This example also shows why it is important to study and understand the behavior of the test statistics using ordering constraints.
%
%


\bibliographystyle{elsarticle-harv}
\bibliography{biblio}





\end{document}